\documentclass[11pt,a4paper]{amsart}
\usepackage{amscd}
\usepackage[T1]{fontenc}
\usepackage[latin1]{inputenc}
\usepackage{enumerate}
\usepackage{amsfonts}
\usepackage{amsthm}
\usepackage{amssymb}
\usepackage[english]{babel}
\usepackage{amsmath,amscd}
\usepackage{amsmath}

\usepackage{hyperref}
\usepackage{xcolor}
\hypersetup{
    colorlinks,
    linkcolor={red!50!black},
    citecolor={blue!50!black},
    urlcolor={blue!80!black}
}

\author[Mathieu Kohli]{Mathieu Kohli$^{1}$}
\thanks{This work has been supported by the ANR projects  ANR-11-LABX-0056-LMH , ANR-15-IDEX-02 and ANR-15-CE40-0018}
\date{\today \\
 $\text{ }\quad\text{ }^{1}$ CMAP Ecole Polytechnique}
\title[Geodesic curvature in the Heisenberg group]{A metric interpretation of the geodesic curvature in the Heisenberg group}
\newtheorem{theo}{Theorem}[section]

\newtheorem{de}[theo]{Definition}

\newtheorem{cor}[theo]{Corollary}
\newtheorem{lem}[theo]{Lemma}

\newtheorem{rem}{Remark}
\let\olddefinition\rem
\renewcommand{\rem}{\olddefinition\normalfont}

\newtheorem{prop}[theo]{Proposition}

\begin{document}
\begin{abstract}In this paper we study the notion of geodesic curvature of smooth horizontal curves parametrized by arc lenght in the Heisenberg group, that is the simplest sub-Riemannian structure. Our goal is to give a metric interpretation of this notion of geodesic curvature as the first corrective term in the Taylor expansion of the distance between two close points of the curve. 
\end{abstract}
\maketitle
\tableofcontents
\section{Introduction}
Since the first half of the nineteenth century and the introduction by Carl Friedrich Gauss of the concept of intrinsic curvature of a surface, several other notions of curvature have been defined. Curvature functions provide a scalar measure of the local geometry around each point of a space or of a geometric object embedded in a space. Our work, that focuses on geodesic curvature of curves in the Heisenberg group, is motivated by \cite{BaloghTysonVecchi16}, \cite{DinizVeloso16} and \cite{ChiuLai13}. Our contribution is to show that the geodesic curvature of a curve corresponds to a measure of a metric property of that same curve. More precisely, we prove that this geodesic curvature appears in the Taylor expansion of the distance between two close points of the curve.
 
Let us recall what happens in the Euclidean case. In this setting, the classical notion of geodesic curvature of a curve $\zeta$ parametrized by arc length at a point $\zeta(t)$ is simply defined as  $\|\zeta''(t)\|$. 
In a Riemannian manifold one can do the same, since it is possible to differentiate the velocity $\zeta'$ along the curve $\zeta$ thanks to the canonical connection $\nabla$, called the Levi-Civita connection that has no torsion and that respects the metric. The previous formula for the geodesic curvature at a point $\zeta(t)$ becomes $\|\nabla_{\zeta'(t)}\zeta'\|$. The geodesic curvature of a curve quantifies its deviation with respect to the Levi-Civita connection.
 
  Notice that in a two dimensional oriented Riemannian surface, we can define the
 \emph{signed} geodesic curvature as $\pm \|\nabla_{\zeta'(t)}\zeta'\|$ where the sign is positive if and only if the frame $\left( \zeta'(t),\nabla_{\zeta'(t)}\zeta' \right)$ is positively oriented. 
This is relevant since if two smooth curves parametrized by arc length have the same initial point, the same initial velocity and the same signed geodesic curvature as a function of time, they are actually the same curve. This result can be interpreted as a ``Frenet-Serret'' theorem in a two dimensional Riemannian space. For more information about Frenet-Serret theory, see for example \cite{manfredo1976carmo}.

To come back to an arbitrary dimension, another perspective on the geodesic curvature of a curve parametrized by arc lenght is the metric one. The key idea is that the geodesic curvature of such a curve is zero along the whole curve if and only if it is a geodesic. In other words, the geodesic curvature is identically zero if and only if the distance between $\zeta(s)$ and $\zeta(s+t)$ is equal to $t$ for every $s$ and for $t$ small enough.  Now for an arbitrary curve parametrized at unit speed, the distance between $\zeta(s)$ and $\zeta(s+t)$ is smaller than $t$ and we expect the correction should depend on the geodesic curvature. As a matter of fact, as a consequence of the expansion of the exponential map, we obtain that for every time $s$,
\begin{align}\label{RiemannianExpansion}
d^2\left( \zeta(s),\zeta(s+t)\right)=t^2-\frac{\|\nabla_{\zeta'(s)}\zeta'\|^2 }{12}t^4 +\mathcal{O}(t^5).
\end{align}

A natural generalization of Riemannian spaces are sub-Riemannian spaces. In order to understand what happens in such spaces, we begin by studying the easiest example, namely the Heisenberg group.
 
 The Heisenberg group $\mathbb{H}$ is $\mathbb{R}^{3}$ whose coordinates we call $x$, $y$ and $z$ endowed with a two dimensional distribution spanned by 
\begin{align*}
X_1 =\frac{\partial}{\partial x}-\frac{y}{2}\frac{\partial}{\partial z}, \qquad X_2=\frac{\partial}{\partial y} +\frac{x}{2}\frac{\partial}{\partial z},
\end{align*}
which we choose to be an orthonormal frame.

A smooth curve $\zeta(t)=(x(t),y(t),z(t))$ that is everywhere tangent to the distribution is said to be horizontal (we will define the notion of horizontality for non necessarily smooth curves afterwards).

 We consider a smooth horizontal curve $\zeta$ such that the norm of $\zeta'$ decomposed on the orthonormal frame $(X_1,X_2)$ is everywhere one and we define the  \emph{characteristic deviation} of $\zeta$ as 
\begin{align}
h_{\zeta}(t)=\dot{x}(t)\ddot{y}(t)-\dot{y}(t)\ddot{x}(t).\label{defintionh}
\end{align}
Now if we write
\begin{align*}
\zeta'(t) =\cos(\theta(t))X_1+\sin(\theta(t))X_2,
\end{align*}
then $\dot{x}(t)=\cos(\theta(t))$ and $\dot{y}(t)=\sin(\theta(t))$ so we can simplify the expression of $h_{\zeta}$ :
\begin{align}
h_{\zeta}(t)=\dot{\theta}(t).\label{expressionhtheta}
\end{align}

We now summarize in a proposition two properties we already mentioned concerning the curvature of curves in the Riemannian case, that are also valid in the Heisenberg group. We prove the following proposition in this paper and it is also possible to recover it from \cite{ChiuLai13}
\begin{prop}\label{hcharacteristicandgeodesic}
\emph{i.}If $\zeta_1:]-T,T[\rightarrow\mathbb{H}$ and $\zeta_2:]-T,T[\rightarrow\mathbb{H}$ are two smooth horizontal curves parametrized by arc length such that for every $t$ in $]-T,T[$, 
\begin{align*}
h_{\zeta_1}(t)=h_{\zeta_2}(t)
\end{align*}
then there exists $\iota$ an isometry of the Heisenberg group such that
\begin{align*}
\zeta_2=\iota\circ\zeta_1.
\end{align*}

\emph{ii.} The derivative of the characteristic deviation $\dot{h}_{\zeta}$ is identically equal to zero along $\zeta$ if and only if $\zeta$ is a geodesic.
\end{prop}  To use the same term as we did in the Riemannian setting, we can say that a ``Frenet-Serret'' characterization of $\zeta$ is given by $\zeta(0)$, ${\zeta}'(0)$ and the knowledge of $h_{\zeta}(t)$ at all times $t$. 

Furthermore, according to the second point of the previous proposition, the quantity
\begin{align*}
 k_{\zeta}(t):=\dot{h}_{\zeta}(t)=\ddot{\theta}(t)
\end{align*}
is called the  \emph{geodesic curvature} of $\zeta$ at time $t$.

 We are interested in the influence of this geodesic curvature on the distance between two close points of the curve we are considering. The main result we prove is the following :
\begin{theo}\label{TheTheorem} If $\zeta:]-T,T[\rightarrow\mathbb{H}$ is a smooth horizontal curve parametrized by arc length in the Heisenberg group then
\begin{align*}
d_{\mathbb{H}}^2(\zeta(0),\zeta(t))=t^2-\frac{\begin{pmatrix}
k_{\zeta}(0)
\end{pmatrix}^2}{720}t^6+\mathcal{O}(t^7).
\end{align*}
\end{theo}
We notice that there is a qualitative jump between what happens in the Riemannian case and the Heisenberg group. Indeed, in the Heisenberg group, the correction in the Taylor expansion of the squared distance between two close points of a curve appears at order six, when in Riemannian spaces, it appears before, at order four.

However, we may wonder if the Taylor expansion of the distance between two points
of a horizontal curve in the Heisenberg group is the limit of the Taylor expansion of that distance in Riemannian structures that ``tend to'' the Heisenberg group ? This is \emph{not} the case if we consider the easiest way in which we can imagine
Riemannian spaces that tend to the Heisenberg group, but there is nevertheless a link between these Taylor expansions. In fact this comes from the interpretation of the characteristic deviation of a curve in the Heisenberg group that was given in \cite{BaloghTysonVecchi16} in terms of Riemannian curvature of the same curve in Riemannian spaces "approximating" the Heisenberg group.

We then complete our brief overview of the characteristic deviation of a curve. We link this deviation to the covariant derivative of the velocity of the curve with respect to $\overline{\nabla}$, the Tanaka-Webster connection in the Heisenberg group. More specifically, we have 
\begin{align*}
|h_{\zeta}(t)|=\left\Vert \overline{\nabla}_{\zeta'(t)}\zeta' \right\Vert_{\mathbb{H}},
\end{align*}
which entails that
\begin{align*}
|k_{\zeta}(t)|=\left| \frac{\text{d}}{\text{d}t} \left\Vert \overline{\nabla}_{\zeta'(t)}\zeta' \right\Vert_{\mathbb{H}}\right|.
\end{align*}
We also notice that the characteristic deviation of a curve corresponds to the Euclidean curvature of the projection of the curve along the $z$-axis on the $(x,y)$-plane. In particular, we find out that curves with constant geodesic curvature
are projected onto so called ``Euler spirals''.

\subsection*{Acknowledgment} I thank my phD advisor Davide Barilari for having presented me with the different aspects of sub-Riemannian curvature that have been studied up to now and encouraged me to work in that direction as well as for having showed me how to write a
paper in a scientific style.

\section{The Heisenberg group}
Here we will quickly present what we need to know about the Heisenberg group. We also refer to \cite{montgomerybook}, \cite{bellaiche}, \cite{ABB} and \cite{riffordbook}.
We have already introduced the Heisenberg group $\mathbb{H}$ as $\mathbb{R}^3$ with coordinates $x$, $y$ and $z$ endowed with a sub-Riemannian structure whose distribution is spanned by the orthonormal frame
\begin{align}\label{TheOrthoFrame}
X_1 =\frac{\partial}{\partial x}-\frac{y}{2}\frac{\partial}{\partial z}\text{,\qquad}X_2=\frac{\partial}{\partial y} +\frac{x}{2}\frac{\partial}{\partial z}.
\end{align}
We also define
\begin{align}\label{TheReebField}
X_3=\frac{\partial}{\partial z}.
\end{align}

We call $g$ the metric on the distribution whose orthonormal frame is $(X_1,X_2)$.

We say that the curve $\zeta :]-T,T[\rightarrow \mathbb{H}$ is horizontal if $\zeta$ is a Lipshitz curve that is almost everywhere tangent to the distribution, whose speed defined with respect to the orthonormal frame $(X_1,X_2)$ is measurable and essentially bounded.

We can compute the length of a horizontal curve by integrating its norm along the curve. The distance between two points is defined as the infimum of length of curves that link those two points. This infimum happens to be a minimum.

We also emphasize the fact that the Heisenberg group is in fact a Lie group on which the sub-Riemannian structure is left-invariant, where the group law $*$ is given by :
\begin{align}\label{multiplication}
\left(x_1,y_1,z_1\right)*\left(x_2,y_2,z_2\right)&=\left(x_1+x_2,y_1+y_2,z_1+z_2+\frac{1}{2}\left(x_1 y_2 - y_1 x_2 \right)\right).
\end{align}
\begin{rem}
In order to study properties of curves that only depend on the sub-Riemannian distance, it is sufficient to consider curves that leave from the origin at time zero, since every other curve can be sent to such a curve by the isometry that corresponds to the left-multiplication by the inverse of the initial point.
\end{rem}
Moreover, we introduce dilations. Specifically we call "dilation centered at 0 of coefficient $r$", where $r$ is a positive real number, the map
\[
\begin{array}{rccl}
\delta_{r}:&\mathbb{H}&\longrightarrow &\mathbb{H}\\
&\left(x,y,z\right)&\longmapsto&\left(rx,ry,r^2z\right).
\end{array}
\]
Dilations preserve the distribution and transform the Heisenberg group's norm $\|\cdot\|_{\mathbb{H}}$ of a horizontal vector $V$  through the following process :
\begin{align}\label{normdilation}
\|{\delta_{r}}_* (V)\|_{\mathbb{H}}=r\| V\|_{\mathbb{H}}.
\end{align}
Moreover, dilations satisfy a certain homogeneity property. For any points $A$ and $B$ in $\mathbb{H}$ :
\begin{align*}
d_{\mathbb{H}}\left(\delta_r(A),\delta_r(B)\right)=rd_{\mathbb{H}}\left(A,B\right).
\end{align*}

Another interesting piece of information about the Heisenberg group is the expression of the geodesics in this space. We recall that a geodesic is a horizontal curve $\gamma:\mathbb{R}\rightarrow\mathbb{H}$ parametrized at constant speed such that for any $t$ in $\mathbb{R}$ and for $s$ in $\mathbb{R}$ close enough to $t$, the lenght of the curve $\gamma$ between times $t$ and $s$ is equal to the distance between $\gamma(t)$ and $\gamma(s)$.

 It is sufficient to give the expression of geodesics parametrized by arc lenght leaving from the origin since the Heisenberg group is a Lie group, it follows that all the other geodesics will be left translations and reparametrizations of these geodesics.
\begin{prop} \label{propgeodesics} A curve $\gamma$ is a geodesic parametrized by arc length leaving from the origin at time zero if, and only if, there exist two real numbers $\omega$ and $\theta_0$ such that the coordinates $\left( x(t),y(t),z(t) \right)$ of $\gamma(t)$ are 
\[\left\lbrace
\begin{array}{rl}
x(t)&=\frac{\sin\left(\omega t +\theta_0\right)-\sin\left( \theta_0\right)}{\omega} \\
y(t)&=\frac{\cos\left( \theta_0\right)-\cos\left(\omega t +\theta_0\right)}{\omega}\\
z(t)&=\frac{1}{2\omega^2}\left( \omega t -\sin\left( \omega t \right) \right),
\end{array}\right.
\]
for $\omega\neq 0$. When $\omega=0$ these formulas become :
\[\left\lbrace
\begin{array}{rl}
x(t)&=t\cos\left( \theta_0\right) \\
y(t)&=t\sin\left(\theta_0\right)\\
z(t)&=0.
\end{array}\right. 
\]
\end{prop}

\section{Main result}
We study ${\zeta}:]-T,T[\rightarrow\mathbb{H}$ a unitary speed smooth horizontal curve leaving from $(0,0,0)$. By using notations we have previously introduced, we can write :
$$\zeta'(t)=\cos(\theta(t))X_1+\sin(\theta(t))X_2,$$
where $\theta$ is a $\mathcal{C}^{\infty}$ smooth function.
In coordinates this means that :
\[
\left \{
\begin{array}{r c l}
\dot{x}(t) &=& \cos(\theta(t)) \\
\dot{y}(t) &=& \sin(\theta(t)) \\
\dot{z}(t) &=& -\frac{y(t)}{2}\cos(\theta(t))+\frac{x(t)}{2}\sin(\theta(t)),

\end{array}
\right .
\]
where $\zeta(t)=(x(t),y(t),z(t))$.
We have the following properties.
\begin{prop}
Let $x(t)$, $y(t)$ and $z(t)$ be the coordinates of $\zeta(t)$. They are $\mathcal{C}^{\infty}$ smooth functions of $t$ and
\begin{align*}
\dot{z}(0)=\ddot{z}(0)=0.
\end{align*}
Moreover, 
\begin{itemize}
\item \emph{Either} for every integer $i\geqslant 1$, ${\theta}^{(i)}(0)=0$ \emph{and in this case} for all integers $j$, $z^{(j)}(0)=0$ and for every $i$ integer greater or equal to two, ${x}^{(i)}(0)=0$ and ${y}^{(i)}(0)=0$.
\item \emph{Or} there exists an integer $i\geqslant 1$ such that  ${\theta}^{(i)}(0)\neq 0$ \emph{which entails that} for $t>0$ close enough to zero,
$\dot{\theta}(t)$ is non-vanishing and the two following identities hold true :
\begin{align}\label{secondformeqdiffz}
 x^2(t)+y^2(t) &=4\int_0^t\int_0^u\left( -\dot{\theta}(s)\dot{z}(s) +\frac{1}{2} \right)\emph{d}s\emph{d}u,\\
 \dddot{z}(t)&=\ddot{\theta}(t)\int_0^t\left( -\dot{\theta}(s)\dot{z}(s) +\frac{1}{2} \right)\emph{d}s-\dot{\theta}^2(t)\dot{z}(t)+\frac{\dot{\theta}(t)}{2}. \label{equadiffzcasethetanonzero}
\end{align}
\end{itemize}
\label{propregularityeqdiff}
\end{prop}
\begin{proof}
First let us notice that $x(0)=y(0)=z(0)=0$ since $\zeta$ leaves from $(0,0,0)$ by definition. The smoothness of $x$ and $y$ with respect to time comes from the fact that $\theta$ is $\mathcal{C}^{\infty}$ smooth and that
\begin{align*}
 x(t)=\int_{0}^t \cos(\theta(s))\text{d}s\qquad\text{ and }\qquad y(t)=\int_{0}^t \sin(\theta(s))\text{d}s.
\end{align*}
The coordinate $z(t)$ is also a smooth function of $t$ according to its expression
\begin{align*}
 z(t)&=\int_{0}^t  -\frac{y(s)}{2}\cos(\theta(s))+\frac{x(s)}{2}\sin(\theta(s)) \text{d}s.
\end{align*}
Then we show that $z$ satisfies a differential equation. 

We start by writing :
\begin{align}\label{expressiondz}
\dot{z}=\frac{x\dot{y}-y\dot{x}}{2} \text{,\qquad}\ddot{x}=-\dot{\theta}\dot{y} \text{,\qquad}\ddot{y}=\dot{\theta}\dot{x}.
\end{align}
In particular, $\dot{z}(0)=0$ and if we differentiate $z$ once more
\begin{align}\label{expressiond2z}
\ddot{z}=\frac{\dot{\theta}\left(x\dot{x}+y\dot{y} \right)}{2}.
\end{align}
This implies that $\ddot{z}(0)=0$ and if we go further in the differentiation
\begin{align*}
\dddot{z}&=\frac{\ddot{\theta}\left(x\dot{x}+y\dot{y} \right)+\dot{\theta}(\overbrace{\dot{x}^2+\dot{y}^2}^{=1}) +\dot{\theta}^2(-x\dot{y}+y\dot{x})}{2}.
\end{align*}
If we multiply this last identity by $\dot{\theta}$ and combine it with (\ref{expressiondz}) and (\ref{expressiond2z}), we obtain
\begin{align}\label{equadiffz}
\dot{\theta}\dddot{z}&=\ddot{\theta}\ddot{z}-\dot{\theta}^3\dot{z}+\frac{\dot{\theta}^2}{2}.
\end{align}

Moreover, (\ref{expressiondz}) and (\ref{expressiond2z}) allow us to assert that if for every integer $i$ greater or equal to one ${\theta}^{(i)}(0)=0$, then for every $i$ integer greater or equal to two, ${x}^{(i)}(0)=0$, ${y}^{(i)}(0)=0$ and ${z}^{(i)}(0)=0$.

On the other hand if we consider $\theta$ such that there exists an integer $i\geqslant 1$ that satisfies $\theta^{(i)}(0)\neq 0$ then for $t>0$ close enough to zero,
$\dot{\theta}(t)$ is non-vanishing
and for such $t$ we can divide the
differential equation (\ref{equadiffz}) by $\dot{\theta}^2(t)$ and find out that :
\begin{align*}
\frac{\partial}{\partial t}\left( \frac{\ddot{z}(t)}{\dot{\theta}(t)} \right)&=-\dot{\theta}(t)\dot{z}(t) +\frac{1}{2}. 
\end{align*}
Therefore the difference between $\frac{\ddot{z}(t)}{\dot{\theta}(t)}$ and $\int_0^t\left( -\dot{\theta}(s)\dot{z}(s) +\frac{1}{2} \right)\text{d}s$ is a constant. But since (\ref{expressiond2z}) holds we know that
\begin{align*}
 \frac{\ddot{z}(t)}{\dot{\theta}(t)}&=\frac{x(t)\dot{x}(t)+y(t)\dot{y}(t) }{2}\stackrel{t\rightarrow 0}{\longrightarrow}0.
\end{align*}
So that for $t>0$ small enough such that $\dot{\theta}(t)\neq 0$ :
\begin{align}
 \frac{\ddot{z}(t)}{\dot{\theta}(t)}&=\int_0^t\left( -\dot{\theta}(s)\dot{z}(s) +\frac{1}{2} \right)\text{d}s.\label{expressionddzdtheta}
\end{align}
But through (\ref{expressiond2z}), we are able to find a second expression for $\frac{\ddot{z}(t)}{\dot{\theta}(t)}$ :
\begin{align*}
 \frac{\ddot{z}(t)}{\dot{\theta}(t)}&=\frac{1}{4}\frac{\partial}{\partial t}\left( x^2(t) +y^2(t) \right).
\end{align*}
As a consequence of the two previous formula, for $t>0$ small enough
\begin{align*}
 \frac{\partial}{\partial t}\left( x^2(t) +y^2(t) \right)&=4\int_0^t\left( -\dot{\theta}(s)\dot{z}(s) +\frac{1}{2} \right)\text{d}s.
\end{align*}
The fact that $\frac{\partial}{\partial t}\left(x^2(t)+y^2(t)\right)$ is continuous and that $x^2(0)+y^2(0)=0$ is sufficient to be sure that for $t$ small enough
\begin{align*}
  x^2(t)+y^2(t) &=4\int_0^t\int_0^u\left( -\dot{\theta}(s)\dot{z}(s) +\frac{1}{2} \right)\text{d}s\text{d}u.
\end{align*}

Finally, still in the case where there exists an integer $i\geqslant 1$ such that $\theta^{(i)}(0)\neq 0$,
we consider $t>0$ small enough to have $\dot{\theta}(t)\neq 0$ and we divide the differential
equation \eqref{equadiffz}  we have already established by $\dot{\theta}(t)$ :
\begin{align*}
\dddot{z}(t)&=\ddot{\theta}(t)\frac{\ddot{z}(t)}{\dot{\theta}(t)}-\dot{\theta}^2(t)\dot{z}(t)+\frac{\dot{\theta}(t)}{2}.
\end{align*}
Then we replace $\frac{\ddot{z}(t)}{\dot{\theta}(t)}$ using \eqref{expressionddzdtheta} and we find out that :
\begin{align*}
\dddot{z}(t)&=\ddot{\theta}(t)\int_0^t\left( -\dot{\theta}(s)\dot{z}(s) +\frac{1}{2} \right)\text{d}s-\dot{\theta}^2(t)\dot{z}(t)+\frac{\dot{\theta}(t)}{2}. 
\end{align*}
\end{proof}

\subsection{Proof of Theorem \ref{TheTheorem}}
We recall that in the Introduction, we stated as Theorem \ref{TheTheorem}, that if $\zeta$ is a smooth horizontal curve parametrized by arc length in the Heisenberg group then
\begin{align*}
d_{\mathbb{H}}^2(\zeta(0),\zeta(t))=t^2-\frac{\begin{pmatrix}
k_{\zeta}(0)
\end{pmatrix}^2}{720}t^6+\mathcal{O}(t^7).
\end{align*}
\begin{proof}
 We know from \cite[Chapter 5, section 5.7. about the Heisenberg group]{AgrachevBarilariRizzi13} that the squared distance between $\zeta(t)$ and the origin, which is also $\zeta(0)$ can be expressed as
\begin{align}\label{expressiondist}
 d_{\mathbb{H}}^2(\zeta(0),\zeta(t))&=\frac{x^2(t)+y^2(t)}{\text{sinc}^2\circ\phi\left( \frac{z(t)}{x^2(t)+y^2(t)} \right)},
\end{align}
where $\phi$ is the inverse function of
\[
 \begin{array}{rccl}
  \psi :&[-\pi,\pi]&\longrightarrow & \mathbb{R}\\
& u &\longmapsto &\frac{1}{4}\left( \frac{u}{\sin^2(u)}-\cot(u) \right).
 \end{array}
\]
We notice that we can rewrite
\begin{align*}
 \psi(u)&=\frac{2u-\sin(2u)}{4(1-\cos(2u))}.
\end{align*}
Then we check that
\begin{align*}
 \psi(u)&=\frac{u}{6}+\frac{u^3}{45}+\mathcal{O}\left( u^5 \right).
\end{align*}
And since $\psi$ is odd and analytic, so is $\phi=\psi^{-1}$ and
\begin{align*}
 \phi(u)&=6u +\alpha u^3 +\mathcal{O}\left( u^5 \right).
\end{align*}
Now
\begin{align*}
 u&=\psi\circ\phi(u)=u +\left(\frac{24}{5} + \frac{\alpha}{6} \right)u^3+\mathcal{O}\left( u^5 \right),
\end{align*}
so $\alpha =-\frac{144}{5}$ and
\begin{align*}
 \phi(u)&=6u  -\frac{144}{5}u^3 +\mathcal{O}\left( u^5 \right).
\end{align*}
We recall that sinc, the cardinal sine function is defined as the entire function such that sinc$(x)=\frac{\sin(x)}{x}$ for all $x$ different from $0$, which implies that
\begin{align*}
\text{sinc}(u)=1-\frac{u^2}{6}+\frac{u^4}{120}+\mathcal{O}\left(u^6\right).
\end{align*}
We are then able to compute
\begin{align}\label{Taylorsinccircphi}
 \frac{1}{\text{sinc}^2\circ\phi(u)}&=1+12u^2-\frac{144}{5}u^4+\mathcal{O}\left(u^5\right).
\end{align}

Now we will need to know the Taylor expansion of $z$ at time zero. We are interested only in the case where there exists an integer $i\geqslant 1$ such that $\theta^{(i)}(0)\neq 0$.
Indeed, 
in the other case, we have already noticed in Proposition \ref{propregularityeqdiff} that for all integers $i$, $z^{(i)}(0)=0$. 
First, by Proposition \ref{propregularityeqdiff}, we have that $z(0)=\dot{z}(0)=\ddot{z}(0)=0$. Then we write (\ref{equadiffzcasethetanonzero})
\begin{align*}
\dddot{z}&=\ddot{\theta}\int_0^t\left( -\dot{\theta}(s)\dot{z}(s) +\frac{1}{2} \right)\text{d}s-\dot{\theta}^2\dot{z}+\frac{\dot{\theta}}{2}.
\end{align*}
We evaluate this identity at zero and find out that :
\begin{align}
\dddot{z}(0)&=\frac{\dot{\theta}(0)}{2}.\label{d3z}
\end{align}
Then we differentiate (\ref{equadiffzcasethetanonzero}) and evaluate the identity we find at zero to obtain
\begin{align}
z^{(4)}(0)&=\ddot{\theta}(0).\label{d4z}
\end{align}
Similarly when we differentiate (\ref{equadiffzcasethetanonzero}) twice and look at what we find at $t=0$ we get
\begin{align}
z^{(5)}(0)&=\frac{3\theta^{(3)}(0)}{2}-\frac{\dot{\theta}^3(0)}{2}.\label{d5z}
\end{align}
These formula for the first differentials of $z$ at zero entail that
\begin{align}\label{expansionz}
z(t)&=\frac{\dot{\theta}(0)}{12}t^3 +\frac{\ddot{\theta}(0)}{24}t^4 +\left( \frac{\theta^{(3)}(0)}{80}-\frac{\dot{\theta}^3(0)}{240} \right)t^5 +\mathcal{O}(t^6).
\end{align}

A last ingredient we will need in order to complete the proof is the expression of the first differentials of $x^2+y^2$ at zero. In order to find these differentials, we use Proposition \ref{propregularityeqdiff} :
\begin{align*}
{x^2(t)+y^2(t)}  &=4\int_0^t\int_0^u\left( -\dot{\theta}(s)\dot{z}(s) +\frac{1}{2} \right)\text{d}s\text{d}u.
\end{align*}
This identity enables us to compute the derivatives of $x^2(t)+y^2(t)$ which we postpone to appendix \ref{derivdistance} and we obtain :
\begin{align}
 {x^2(t)+y^2(t)}  &={t^2}-\frac{\dot{\theta}^2(0)}{12}t^4-\frac{\dot{\theta}(0)\ddot{\theta}(0)}{12}t^5 \nonumber\\
&+\left(-\frac{\dot{\theta}(0)\dddot{\theta}(0)}{40}-\frac{\ddot{\theta}^2(0)}{45}+\frac{\dot{\theta}^4(0)}{360}  \right)t^6 +\mathcal{O}(t^7). \label{expressionrcarre}
\end{align}
\begin{rem}
The expansions that are given by (\ref{expansionz}) and (\ref{expressionrcarre}) are still valid in the case where
for all integers $i\geqslant 1$, $\theta^{(i)}(0)=0$, according to the first point in Proposition \ref{propregularityeqdiff}.
\end{rem}

Now we can combine (\ref{expressiondist}), (\ref{Taylorsinccircphi}), (\ref{expansionz}) and (\ref{expressionrcarre}) to obtain
\begin{align*}
d_{\mathbb{H}}^2(\zeta(0),\zeta(t))=t^2-\frac{\begin{pmatrix}
\ddot{\theta}^2(0)
\end{pmatrix}^2}{720}t^6+\mathcal{O}(t^7).
\end{align*}
\end{proof}

\section{Proof of Proposition \ref{hcharacteristicandgeodesic}}
In the introduction, we stated in Proposition \ref{hcharacteristicandgeodesic} that for $\zeta$ a horizontal curve parametrized by arc length, the function $h_{\zeta}$ characterizes $\zeta$ up to isometry and that its derivative is identically zero if and only if $\zeta$ is a geodesic. This was already noticed in \cite{ChiuLai13} but let us give a quick proof. We need two definitions and a lemma :
\begin{de} For $u\in\mathbb{H}$ we define the left translation by $u$ :
\[
\begin{array}{rccc}
L_{u} :& \mathbb{H}& \longrightarrow & \mathbb{H}\\
& v &\longmapsto & u*v.
\end{array}
\]
We also define for any real number $\alpha$, $R_{\alpha}$ the rotation around the $z$-axis on $\mathbb{H}$, namely
\begin{align*}
 R_{\alpha}=
 \begin{pmatrix}
           \cos(\alpha)& \sin(\alpha) & 0 \\
           -\sin(\alpha)& \cos(\alpha) &0\\
           0& 0 &1 
           \end{pmatrix}.
\end{align*}

\end{de}
Both of the previous definitions are relevant because of the following lemma.
\begin{lem}\label{lemmaactionRL} For any real number $\alpha$, $R_{\alpha}$ is an isometry that preserves $(X_1,X_2,X_3)$. Moreover, if $\zeta:]-T,T[\rightarrow\mathbb{H}$ is a smooth horizontal curve parametrized by arc lenght and $u\in\mathbb{H}$,
\begin{align*}
h_{\zeta}=h_{R_{\alpha}\circ\zeta} \text{ and }h_{\zeta}=h_{L_u \circ \zeta}.
\end{align*}
\end{lem}
\begin{proof}
The fact that $R_{\alpha}$ preserves $(X_1,X_2,X_3)$ comes from a computation in coordinates. Since $R_{\alpha}$ preserves $(X_1,X_2,X_3)$, it is an isometry. Moreover, the fact that $R_{\alpha}$ preserves $(X_1,X_2,X_3)$ means that the angle $\theta(t)$ between $\zeta'(t)$ and $X_1$ is left invariant by $R_{\alpha}$ so by using \eqref{expressionhtheta}, we obtain 
\begin{align*}
h_{\zeta}=h_{R_{\alpha}\circ\zeta}.
\end{align*}
To get $h_{\zeta}=h_{L_u \circ \zeta}$ we combine \eqref{defintionh} and \eqref{multiplication}. 
\end{proof}

We are now ready to prove Proposition \ref{hcharacteristicandgeodesic}.
\begin{proof}[Proof of Proposition \ref{hcharacteristicandgeodesic}]
Let us assume that $\zeta_1$ and $\zeta_2$ are two smooth horizontal curves parametrized by arc length that are defined on a same time interval and that for times $t$ in their domain,
\begin{align*}
h_{\zeta_1}(t)=h_{\zeta_2}(t)
\end{align*}
There exists an angle $\alpha_0$ such that $L_{{\zeta_1}(0)^{-1}}\circ {\zeta_1}$ and $R_{\alpha_0}\circ L_{{\zeta_2}(0)^{-1}}\circ {\zeta_2}$ are two curves that start at the origin with the same initial velocity
\begin{align*}
\cos\left(\theta_0\right)X_1 +\sin\left(\theta_0\right)X_2.
\end{align*}
But according to \eqref{expressionhtheta}, the velocity of $L_{{\zeta_1}(0)^{-1}}\circ\zeta_1$ at time $t$ is :
\begin{align*}
&\cos\left(\theta_0 +\int_0^t h_{L_{{\zeta_1}(0)^{-1}}\circ\zeta_1}\text{d}t\right)X_1 +\sin\left(\theta_0 +\int_0^t h_{L_{{\zeta_1}(0)^{-1}}\circ\zeta_1}\text{d}t\right)X_2\\
&\quad =\cos\left(\theta_0 +\int_0^t h_{\zeta_1}\text{d}t\right)X_1 +\sin\left(\theta_0 +\int_0^t h_{\zeta_1}\text{d}t\right)X_2,\text{ by Lemma \ref{lemmaactionRL}.}\\
&\quad =\cos\left(\theta_0 +\int_0^t h_{\zeta_2}\text{d}t\right)X_1 +\sin\left(\theta_0 +\int_0^t h_{\zeta_2}\text{d}t\right)X_2\\
&\quad =\cos\left(\theta_0 +\int_0^t h_{R_{\alpha_0}\circ L_{{\zeta_2}(0)^{-1}}\circ {\zeta_2}}\text{d}t\right)X_1\\
&\quad\quad +\sin\left(\theta_0 +\int_0^t h_{R_{\alpha_0}\circ L_{{\zeta_2}(0)^{-1}}\circ {\zeta_2}}\text{d}t\right)X_2, \text{ thanks to Lemma \ref{lemmaactionRL}.}
\end{align*}
By using the identity \eqref{expressionhtheta}, we notice that this last vector is equal to the velocity of $R_{\alpha_0}\circ L_{{\zeta_2}(0)^{-1}}\circ {\zeta_2}$ at time $t$.

Since $L_{{\zeta_1}^{-1}(0)}\circ {\zeta_1}$ and $R_{\alpha_0}\circ L_{{\zeta_2}^{-1}(0)}\circ {\zeta_2}$ both start at the origin and are integral lines of the same vector fields they are in fact the same curve and it follows that $\zeta_1$ and $\zeta_2$ are equal up to an isometry.

Furthermore, we know by Proposition \ref{propgeodesics} that a smooth horizontal curve $\zeta$ that is parametrized by arc length and leaves from the origin at time zero is a geodesic if and only the angle $\theta$ it forms with $(X_1,X_2)$ is an affine function of time, which means that $h_{\zeta}$ is a constant. Now for $\zeta$ leaving from any point,
\begin{align*}
\zeta \text{ is a geodesic } & \text{ if and only if } L_{{\zeta}^{-1}(0)}\circ {\zeta} \text{ is a geodesic,}\\
&\text{ if and only if } h_{L_{{\zeta}^{-1}(0)}\circ {\zeta}} \text{ is constant,}\\
&\text{ if and only if } h_{ {\zeta}} \text{ is constant.}\\
\end{align*}

\end{proof}

\section{Final remarks}
At this point, we intend to emphasize the fact that our interpretation of the geodesic curvature of a curve is linked to the curvature defined in \cite{BaloghTysonVecchi16} as the common Riemannian curvature of the curve in spaces that tend to the Heisenberg space. 

More explicitly, for $\varepsilon>0$, we consider the $\varepsilon-$Riemannian structures on the Heisenberg group such that $(X_1,X_2,\varepsilon X_3)$ is an orthonormal frame, where we recall that the vector fields $X_i$ are defined in \eqref{TheOrthoFrame} and \eqref{TheReebField}.

We denote by $g_{\varepsilon}$ the metric on the $\varepsilon$-Riemannian structure, by $\Vert . \Vert_{\varepsilon}$ its norm, by $d_{\varepsilon}(.,.)$ the distance function on this structure and by $\nabla^{\varepsilon}$ the associated Levi-Civita connection.
These $\varepsilon-$Riemannian structures converge in the pointed Gromov-Haussdorff sense to the Heisenberg sub-Riemannian structure as $\varepsilon$ goes to zero (see for example \cite{bellaiche}).

Corollary \ref{hepsiloncurvature} tells us that the Riemmannian curvature of $\zeta$ in these various $\varepsilon-$Riemannian structures does not depend on $\varepsilon$. We prove this result here and another proof of this fact is contained in \cite{BaloghTysonVecchi16}. In that same paper, they choose to call this common Riemannian curvature the sub-Riemannian curvature of the curve.
In our own vocabulary, this corresponds to the characteristic deviation.
\begin{prop} The following identities are satisfied :
\begin{align*}
\nabla^{\varepsilon}_{X_1} X_1 =0,\qquad\nabla^{\varepsilon}_{X_2} X_2 =0,\qquad\nabla^{\varepsilon}_{X_1} X_2 = \frac{X_3}{2}=-\nabla^{\varepsilon}_{X_2} X_1.
\end{align*}\label{constantepsiloncurvature} 
\end{prop}
Before we prove this proposition we give the corollary we mentioned a few lines ago :
\begin{cor}\label{hepsiloncurvature} If $\zeta:]-T,T[\rightarrow\mathbb{H}$ is a smooth horizontal curve parametrized by arc length that forms at time $t$ an angle $\theta(t)$ with $X_1$ then
\begin{align*}
\nabla^{\varepsilon}_{\zeta'(t)}\zeta' =h_{\zeta}(t)\left( -\sin(\theta(t))X_1\left( \zeta(t) \right) +\cos(\theta(t))X_2\left( \zeta(t) \right) \right).
\end{align*}
In particular,
\begin{align*}
\Vert \nabla^{\varepsilon}_{\zeta'(t)}\zeta' \Vert_{\varepsilon} =|h_{\zeta}(t)|.
\end{align*}\label{curvatureepsilonspaces}
\end{cor}
\begin{proof}[Proof of Proposition \ref{constantepsiloncurvature}]
We recall that for $X$, $Y$ and $Z$ three vector fields :
\begin{align*}
 2 g\left(\nabla_X Y,Z\right)=&X g\left(Y,Z \right) +Y g\left( Z,X \right) -Z g\left( X,Y \right)\\
& +g\left( \left[ X,Y\right], Z\right) +g\left( \left[ Z,X\right], Y\right)-g\left( \left[ Y,Z\right], X\right).
\end{align*}
Now if we consider the $\varepsilon$-Riemannian structure and we choose $X$, $Y$, and $Z$ among $X_1$, $X_2$ and $\varepsilon X_3$, this Koszul identity is reduced to 
\begin{align*}
 2 g_{\varepsilon}\left(\nabla^{\varepsilon}_X Y,Z\right)= g_{\varepsilon}\left( \left[ X,Y\right], Z\right) +g_{\varepsilon}\left( \left[ Z,X\right], Y\right)-g_{\varepsilon}\left( \left[ Y,Z\right], X\right).
\end{align*}
Moreover, we notice that in the case where $Y=Z$ we simply get 
\begin{align*}
 g_{\varepsilon}\left(\nabla^{\varepsilon}_X Y,Y\right)&=0.
\end{align*}
Another usefull remark is that for all $i$, $\left[X_i,\varepsilon X_3\right]=\varepsilon\left[X_i, X_3\right] =0$.\\
This allows us to write 
\begin{align*}
 g_{\varepsilon}\left(\nabla^{\varepsilon}_{X_1} X_1,X_1\right)=g_{\varepsilon}\left(\nabla^{\varepsilon}_{X_1} X_1,\varepsilon X_3\right) =0,
\end{align*}
and
\begin{align*}
 g_{\varepsilon}\left(\nabla^{\varepsilon}_{X_1} X_1,X_2\right)&=-g_{\varepsilon}\left( \left[ X_1,X_2 \right], X_1 \right)\\
 &=-g_{\varepsilon}\left(X_3, X_1 \right)\\
 &=0.
\end{align*}
So $\nabla^{\varepsilon}_{X_1} X_1=0$. The same way we prove that $\nabla^{\varepsilon}_{X_2} X_2=0$.\\
We must also compute
\begin{align*}
  g_{\varepsilon}\left(\nabla^{\varepsilon}_{X_1} X_2,X_1\right)&=g_{\varepsilon}\left( \left[ X_1,X_2 \right], X_1 \right)\\
  &=g_{\varepsilon}\left( X_3, X_1 \right)\\
  &=0,
\end{align*}
as well as
\begin{align*}
 2 g_{\varepsilon}\left(\nabla^{\varepsilon}_{X_1} X_2,\varepsilon X_3\right)&=g_{\varepsilon}\left(\left[X_1,X_2\right],\varepsilon X_3\right)=g_{\varepsilon}\left(X_3,\varepsilon X_3\right)=\frac{1}{\varepsilon},
\end{align*}
which implies that
\begin{align*}
 \nabla^{\varepsilon}_{X_1} X_2 &= \frac{X_3}{2}.
\end{align*}
But by the torsion-freedom of the Levi-Civita connection $\nabla^{\varepsilon}_{X_1} X_2 -\nabla^{\varepsilon}_{X_2} X_1=\left[ X_1 , X_2  \right]= X_3$ so
\begin{align*}
 \nabla^{\varepsilon}_{X_2} X_1 &= -\frac{X_3}{2}.
\end{align*}
\end{proof}

Since the main result of this paper is the Taylor expansion in Theorem \ref{TheTheorem}, a natural question is : can we compare it to the same expansion in the $\varepsilon-$Riemannian structures ? Indeed we can, by combining Corollary \ref{curvatureepsilonspaces} and \eqref{RiemannianExpansion}. We denote by $d_\varepsilon$ the distance in the $\varepsilon-$Riemannian structure and we obtain :
\begin{cor}If $\zeta:]-T,T[\rightarrow\mathbb{H}$ is a smooth horizontal curve parametrized by arc length then
\begin{align*}
d_{\varepsilon}^2\left( \zeta(0),\zeta(t)\right)=t^2-\frac{\left(  h_{\zeta}(0)\right)^2 }{12}t^4 +\mathcal{O}(t^5).
\end{align*}
\end{cor}

Now we can also give a proposition similar to Corollary \ref{curvatureepsilonspaces} but where the connection that appears actually is linked to the Heisenberg structure and not to a Riemannian approximation. Indeed, in the Heisenberg group, we can define the Tanaka-Webster connection. Before we do this, we recall that the vector fields $X_i$ are defined by \eqref{TheOrthoFrame} and \eqref{TheReebField}.

\begin{de} The Tanaka-Webster connection $\overline{\nabla}$ in the Heisenberg group is the connection, such that :\\

$\bullet$ the Reeb vector field $X_3$ is parallel with respect to $\overline{\nabla}$,\\

$\bullet$ for any horizontal vector field $X$ and any vector field $Y$, $\overline{\nabla}_Y X$ is horizontal,\\

$\bullet$ $\overline{\nabla}g=0$,\\

$\bullet$ the torsion of any two horizontal vector fields is colinear to $X_3$,\\

$\bullet$ for any $i\in\{ 1,2,3\}$, $\overline{\nabla}_{X_3}X_i=0$.
\end{de}
We now link this connection and the characteristic deviation in the Heisenberg group as we did before in the $\varepsilon-$Riemannian structures :
\begin{prop} For $(i,j)\in\lbrace 1,2\rbrace^2$ : $\overline{\nabla}_{X_i} X_j=0$.
\end{prop}
Which entails that :
\begin{cor}
If $\zeta:]-T,T[\rightarrow\mathbb{H}$ is a smooth horizontal curve parametrized by arc length that forms at time $t$ an angle $\theta(t)$ with $X_1$ then
\begin{align*}
\overline{\nabla}_{\zeta'(t)}\zeta' =h_{\zeta}(t)\left( -\sin(\theta(t))X_1\left( \zeta(t) \right) +\cos(\theta(t))X_2\left( \zeta(t) \right) \right).
\end{align*}
In particular,
\begin{align*}
\left\Vert \overline{\nabla}_{\zeta'(t)}\zeta' \right\Vert_{\mathbb{H}} =|h_{\zeta}(t)|.
\end{align*}

\end{cor}
\begin{proof}
If we write $T(X,Y)$ for the torsion with respect to the Tanaka-Webster connection of two vector fields $X$ and $Y$ then
\begin{align*}\overline{\nabla}_{X_1}X_2 -\overline{\nabla}_{X_2}X_1 =\left[ X_1, X_2 \right] +T(X_1,X_2)=X_3+T(X_1,X_2).
\end{align*}
By the defintion of $\overline{\nabla}$, the left hand side of the previous identity is in the distribution and the right hand side is colinear to $X_3$. As a consequence, both these quantities vanish and
\begin{align*}
\overline{\nabla}_{X_1}X_2 =\overline{\nabla}_{X_2}X_1.
\end{align*}
And since these vector fields are horizontal, we can rewrite this identity as :
\begin{align}
g\left( \overline{\nabla}_{X_1}X_2,X_1\right)=g\left(\overline{\nabla}_{X_2}X_1,X_1 \right)\text{ and }g\left( \overline{\nabla}_{X_1}X_2,X_2\right)=g\left(\overline{\nabla}_{X_2}X_1,X_2 \right).\label{nablag1}
\end{align}
Moreover, we know by definition of the Tanaka-Webster connection that $\overline{\nabla}g=0$, therefore we can transform all the equations
\begin{align*}
X_ig(X_j,X_k)=0 \text{ for $(i,j,k)\in\lbrace 1,2\rbrace^3$,}
\end{align*}  into
\begin{align}
g\left(\overline{\nabla}_{X_1}X_1 ,X_1  \right)= g\left(\overline{\nabla}_{X_2}X_1 ,X_1  \right) =g\left(\overline{\nabla}_{X_1}X_2 ,X_2  \right)=g\left(\overline{\nabla}_{X_2}X_2 ,X_2  \right)=0,\label{nablag2}
\end{align}
and
\begin{align}
g\left(\overline{\nabla}_{X_1}X_1 ,X_2  \right)+g\left(\overline{\nabla}_{X_1}X_2 ,X_1  \right)=0,\label{nablag3}
\end{align}
as well as
\begin{align}
g\left(\overline{\nabla}_{X_2}X_1 ,X_2  \right)+g\left(\overline{\nabla}_{X_2}X_2 ,X_1  \right)=0,\label{nablag4}
\end{align}
Now if we solve the system that comes from \eqref{nablag1}, \eqref{nablag2}, \eqref{nablag3} and\eqref{nablag4} we find out that for $(i,j,k)\in\lbrace 1,2\rbrace^3$ we have
\begin{align*}
g\left(\overline{\nabla}_{X_i}X_j ,X_k  \right)=0.
\end{align*}
From which we deduce that for $(i,j)\in\lbrace 1,2\rbrace^2$ : $\overline{\nabla}_{X_i} X_j=0$.\\

\end{proof}

A last interpretation of the characteristic deviation of a curve in the Heisenberg space comes from the Euclidean curvature of the projection of the curve on the $(x,y)$-plane. It is not too far-fetched since the Heisenberg group can be constructed in the first place as a convenient extension of the $(x,y)$-plane to solve the isoperimetric problem (see for example \cite{ABB}). 

 We denote by $\pi$ the projection onto the $(x,y)$-plane defined as
$\pi:\mathbb{H}\to \mathbb{R}^2$ such that  $\pi(x,y,z):=(x,y)$.

\begin{prop}\label{propcharacteristicdevprojEuclidcurv} For $\zeta:]-T,T[\rightarrow\mathbb{H}$ a smooth horizontal curve parametrized by arc length, its characteristic deviation at time $t$, $h_{\zeta}(t)$  is equal to the Euclidean curvature of $\pi\circ\zeta$ at time $t$.
\end{prop}
\begin{proof}
The projection of a smooth horizontal curve $\zeta$ parametrized by arc length along the $z$ axis on the plane $(x,y)$ (that we can endow with the canonical Euclidean structure on $\mathbb{R}^2$) is a curve parametrized by arc length. Indeed $\dot{x}^2+\dot{y}^2=1$ as a consequence of the fact that $\zeta$ is parametrized by arc length in the Heisenberg group. But the expression of the signed curvature of a curve at time $t$ in the Euclidean plane is $$\frac{\dot{x}(t)\ddot{y}(t)-\dot{y}(t)\ddot{x}(t)}{\left( \dot{x}^2(t)+\dot{y}^2(t) \right)^{\frac{3}{2}}}.$$ That means that in the case we are considering $h_{\zeta}$ that is defined in (\ref{defintionh}) is equal to the curvature in the Euclidean plane $(x,y)$ of the projection of $\zeta$ along $z$.
\end{proof}
In particular, we obtain the following corollary :
\begin{cor} The projection along the $z$ axis on the $(x,y)$ plane of the trajectories of curves with constant geodesic curvature are so-called "Euler spirals", which are, up to rotations, translations, symmetries and dilations (followed by an affine reparametrization to keep a unitary speed) no more than the trajectory given by
\[\begin{array}{rcl}
\mathbb{R}&\longrightarrow & \mathbb{R}\\
t&\longmapsto & \begin{pmatrix}
\int_0^t \cos(u^2) \text{d}u \\
\int_0^t \sin(u^2) \text{d}u
\end{pmatrix}.
\end{array}
\]
\end{cor}
\begin{proof}
For $\zeta$ a smooth horizontal curve parametrized by arc length, its geodesic curvature $k_{\zeta}$ is by definition the derivative of $h_{\zeta}$ which is itself the Euclidean curvature of the projection of $\zeta$ on the plane $(x,y)$ according to Proposition \ref{propcharacteristicdevprojEuclidcurv}. Therefore, the curves that have no geodesic curvatures are projected along $z$ onto circles of the plane $(x,y)$, while the curves with constant geodesic curvature in the Heisenberg group are projected onto curves with affine Euclidean curvature.
But curves with affine Euclidean curvature have their velocity that forms a quadratic angle with a fixed direction.

In coordinates, we can write the projection through $\pi$ of any curve $\zeta$ with constant geodesic curvature in the Heisenberg group as :
\begin{align*}
\pi\circ\zeta(t)&=\begin{pmatrix}
\int_{0}^t\cos\left(\pm\left(\left( as +b \right)^2 +c \right)\right) \text{d}s \\
\int_{0}^t\sin\left(\pm\left(\left( as +b \right)^2 +c \right)\right) \text{d}s 
\end{pmatrix}.
\end{align*}
Therefore an arbitrary curve with constant geodesic curvature is projected through $\pi$ onto :
\begin{align*}
\frac{1}{a}\begin{pmatrix} 1 & 0\\
0 &\pm 1
\end{pmatrix}
\begin{pmatrix}
\cos(c) & -\sin(c)\\
\sin(c) & \cos(c)
\end{pmatrix}\begin{pmatrix}
\int_{0}^{at+b}\cos\left(u^2  \right) \text{d}u \\
\int_{0}^{at+b}\sin\left( u^2 \right) \text{d}u 
\end{pmatrix}.
\end{align*}
\end{proof}
\begin{rem} Euler spirals have been extensively studied. It is possible to find a history and important properties of those curves in \cite{levien2008euler}.
\end{rem}

Another property of the curvature we might be interested in is : how is it transformed by the action of dilations ? We give the answer in the following proposition
\begin{prop} We consider $\zeta:]-T,T[\rightarrow\mathbb{H}$ a smooth horizontal curve parametrized by arc length. For $r>0$, its dilation
$$\xi_{r}(t):=\delta_{r}\circ \zeta\left( \frac{t}{r} \right)$$
is horizontal and parametrized by arc length. Moreover the geodesic curvature of the dilated curve is linked to the geodesic curvature of the initial curve by the relation :
$$k_{\xi_r}(rt)=\frac{1}{r^2}k_{\zeta}(t).$$
\end{prop}
\begin{proof}
We already stated as a general property of dilations that they preserve horizontal curves and thanks to \eqref{normdilation} we learn that for $\zeta$ a horizontal curve parametrized by arc length and $r$ positive, the curve $\xi_{r}$  is a horizontal curve parametrized by arc length. Now we remember that according to \ref{propcharacteristicdevprojEuclidcurv}, the characteristic deviation of a smooth horizontal curve parametrized by arc length is simply the curvature of its projection along $z$ and since dilations of the Heisenberg group act as usual Euclidean dilations when projected on $(x,y)$ we find out that $h_{\xi_r}(rt)=\frac{1}{r}h_{\zeta}(t)$ which implies by definition of the geodesic curvature, that $k_{\xi_r}(rt)=\frac{1}{r^2}k_{\zeta}(t)$.
\end{proof}

\appendix

\section{Proof of the identity \eqref{expressionrcarre}}
We recall that according to Proposition \ref{propregularityeqdiff} :
\begin{align*}
{x^2(t)+y^2(t)}  &=4\int_0^t\int_0^u\left( -\dot{\theta}(s)\dot{z}(s) +\frac{1}{2} \right)\text{d}s\text{d}u.
\end{align*}
As a consequence 
\begin{align*}
{x^2(0)+y^2(0)}=\frac{\partial}{\partial t} \Big|_{t=0}\left({x^2(t)+y^2(t)}\right)=0\text{,\qquad}
\end{align*}
\begin{align*}
 \frac{\partial^2}{\partial t^2} \Big|_{t=0}\left({x^2(t)+y^2(t)}\right)={2} \text{ (since $\dot{z}(0)=0$ by Proposition \ref{propregularityeqdiff}),}
\end{align*}
and for $n\geqslant 3$ :
\begin{align*}
\frac{\partial^n}{\partial t^n}\left( {x^2(t)+y^2(t)} \right)&=-4\sum_{i=0}^{n-2}\begin{pmatrix}
n-2\\
i
\end{pmatrix} \theta^{(i+1)}z^{(n-i-1)}.
\end{align*}
Now we remember that $\dot{z}(0)=\ddot{z}(0)=0$ by Proposition \ref{propregularityeqdiff}) and that $z^{(3)}(0)$, $z^{(4)}(0)$ and $z^{(5)}(0)$ are given by (\ref{d3z}), (\ref{d4z}) and (\ref{d5z}) so we get
\begin{align*}
\frac{\partial^3}{\partial t^3}\Big|_{t=0}\left( {x^2(t)+y^2(t)} \right)=0,
\end{align*}
and
\begin{align*}
\frac{\partial^4}{\partial t^4}\Big|_{t=0}\left( {x^2(t)+y^2(t)} \right)&=-4\dot{\theta}(0)z^{(3)}(0)
=-2{\dot{\theta}^2(0)},
\end{align*}
and also
\begin{align*}
\frac{\partial^5}{\partial t^5}\Big|_{t=0}\left( {x^2(t)+y^2(t)} \right)&=-12\ddot{\theta}(0)z^{(3)}(0)-4\dot{\theta}(0)z^{(4)}(0)=-10\dot{\theta}(0)\ddot{\theta}(0),
\end{align*}
and finally
\begin{align*}
\frac{\partial^6}{\partial t^6}\Big|_{t=0}\left( {x^2(t)+y^2(t)} \right)&=-24\dddot{\theta}(0)z^{(3)}(0)-16\ddot{\theta}(0)z^{(4)}(0)-4\dot{\theta}(0)z^{(5)}(0)\\
&=-18\dot{\theta}(0)\dddot{\theta}(0)-16\ddot{\theta}^2(0) +2{\dot{\theta}^4(0)}.
\end{align*}
\label{derivdistance}

\bibliographystyle{alpha}
\bibliography{biblioheis}

\end{document}